\documentclass{siamart190516}

\usepackage{amsmath,amssymb,amsfonts,stmaryrd}

\usepackage{lineno,hyperref}
\usepackage{pdfsync}
\usepackage[T1]{fontenc}
\usepackage{xspace}
\usepackage{ulem}
\usepackage{sidecap}
\usepackage{cancel}
\usepackage{amsmath, amssymb, stmaryrd}
\usepackage{afterpage}
\newcommand{\vertiii}[1]{{\left\vert\kern-0.25ex\left\vert\kern-0.25ex\left\vert #1 
    \right\vert\kern-0.25ex\right\vert\kern-0.25ex\right\vert}}

\modulolinenumbers[5]
\graphicspath{{./figures/}{./tikz_figures/}}

\usepackage{xcolor}
\usepackage{booktabs}
\usepackage{pgfplots}
\usepackage{pgfplotstable}
\usepackage{tikz}
\usepgfplotslibrary{external}
\usepackage{mytikz-private}
\usepackage{subcaption}

\newcommand{\Smooth}{\ensuremath{\mathcal{S}\!mooth}\xspace}
\newcommand{\Prolongate}{\ensuremath{\mathcal{P}\!rolongate}\xspace}
\newcommand{\Solve}{\ensuremath{\mathcal{S}\!olve}\xspace}
\newcommand{\Estimate}{\ensuremath{\mathcal{E}\!stimate}\xspace}
\newcommand{\Mark}{\ensuremath{\mathcal{M}ark}\xspace}
\newcommand{\Refine}{\ensuremath{\mathcal{R}efine}\xspace}
\newcommand{\dealii}{\texttt{deal.II}\xspace}

\pgfplotsset{compat=1.14}

\title{Quasi-optimal mesh sequence construction through
  Smoothed Adaptive Finite Element Methods}

\headers{S-AFEM}{O.Mulita, S.Giani, and L.Heltai}

\author{Ornela Mulita\thanks{Corresponding Author: \email{omulita@sissa.it}, Scuola Internazionale Superiore di Studi Avanzati, Via Bonomea 265, 34136 Trieste, IT}
  \and
  Stefano Giani\thanks{\email{stefano.giani@durham.ac.uk}, Durham University, Stockton Road, Durham, DH1 3LE UK}
  \and
  Luca Heltai \thanks{\email{luca.heltai@sissa.it}, Scuola Internazionale Superiore di Studi Avanzati, Via Bonomea 265, 34136 Trieste, IT}}

\begin{document}
\maketitle

\begin{abstract}
We propose a new algorithm for Adaptive Finite Element Methods (AFEMs) based on smoothing iterations (S-AFEM), for linear, second-order, elliptic partial differential equations (PDEs). The algorithm is inspired by the ascending phase of the V-cycle multigrid method: 
we replace accurate algebraic solutions in intermediate cycles of the classical AFEM with the application of a prolongation step, followed by the application of a smoother.
Even though these intermediate solutions are far from the exact algebraic solutions, their a-posteriori error estimation produces a refinement pattern that is substantially equivalent to the one that would be generated by classical AFEM, at a considerable fraction of the computational cost. 

We provide a qualitative analysis of how the error propagates throughout the algorithm, and we present a series of numerical experiments that highlight the efficiency and the computational speedup of S-AFEM.
\end{abstract}



\section{Introduction}
The efficient numerical simulation of complex real-world phenomena requires the use of computationally affordable discrete models. 
The adaptive finite element method (AFEM) is one such a scheme  for the numerical solution of partial differential equations (PDEs) in computational sciences and engineering.
In finite element simulations (FEM), the domain of the PDE is discretized into a large set of small and simple domains (the cells or elements) depending on a size parameter $h>0$, and the PDE is transformed into an algebraic system of equations.
Rigorous analysis of the numerical method allows one to estimate the discretization error both a priori (giving global bounds on the total discretization error that depend on a global size parameter h), and a posteriori (providing a local distribution of the error on the discretized mesh in terms of known quantities). 
Classical AFEM consists of successive loops of the steps 
 $
\Solve \longrightarrow \Estimate \longrightarrow \Mark \longrightarrow  \Refine
 $
 to decrease the total discretization error, by repeating the FEM solution process (\Solve) on a mesh that has been refined (\Refine) where the a-posteriori analysis (\Estimate) has shown that the error is larger (\Mark).

Intermediate solution steps are \textit{instrumental} for the construction of the finally adapted grid, and play no role in the final solution, which is the only one that is retained for analysis and processing.

In this work we present a simple yet effective algorithm to reduce the overall computational cost of the AFEM algorithm, by providing a fast procedure for the construction of quasi-optimal mesh sequences that do not require the exact solution of the algebraic problem in the intermediate loops of AFEM. 

We refer to this new algorithm as \textit{Smoothed Adaptive Finite Element} (S-AFEM) method: the \Solve step of AFEM in all intermediate loops is replaced by the application of a prolongation step (\Prolongate), followed by the application of a smoother (\Smooth):

\tikzset{
  block/.style = {fill=white, minimum height=3em, minimum width=3em},
  }
 \begin{center}
  \begin{tikzpicture}[auto, node distance=2cm,>=latex']
    \node [block, name=firstsolve] (firstsolve) {\Solve};
    \node [block, right of=firstsolve] (estimate) {\Estimate};
    \node [block, right of=estimate] (mark) {\Mark};
    \node [block, right of=mark] (refine) {\Refine};
    \node [block, below of=refine, node distance=2.3cm] (prolongate){\Prolongate};
    \node [block, below of=estimate, node distance=2.3cm] (smooth){\Smooth};
    \node [block, right of=refine] (lastsolve) {\Solve};
    \draw [->] (firstsolve) -- (estimate);
    \draw [->] (estimate) -- (mark);
    \draw [->] (mark) -- (refine);
    \draw [->] (refine) -- (lastsolve);
    \draw [->] (refine) -- (prolongate);
    \draw [->] (prolongate) -- (smooth);
    \draw [->] (smooth)--(estimate);
  \end{tikzpicture}
 \end{center}

S-AFEM takes its inspiration from the ascending phase of the V-cycle multigrid method where a sequence of prolongation and smoothing steps is applied to what is considered an algebraically exact solution at the coarsest level. In the multigrid literature, this procedure is used to transfer the low frequency information contained in the coarse solution to a finer --nested-- grid, where some steps of a smoothing iteration are applied in order to improve the accuracy of the solution in the high frequency range (see, for example, the classical references~\cite{hackbusch1994iterative, hackbusch2013multi, xu1992iterative, bramble2000analysis, bramble2018multigrid}). The iteration of this procedure is very effective in providing accurate algebraic solutions in $O(N)$ time, where $N$ is the dimension of the final algebraic system. Even a small number of smoothing iterations is sufficient to eliminate the high frequency error, while the prolongation from coarser grids guarantees the convergence in the low frequency regime, resulting in an overall accurate solution, also when local refinement is present (see, for example,~\cite{janssen2011adaptive}).

The classical AFEM algorithm generates nested grids and subspaces\footnote{During AFEM, the grids remain nested if no de-refinement occurs. In the following we will work under this assumption.}, but the construction of the next (unknown) grid in the sequence still requires an exact algebraic solution on the current grid to trigger the \Estimate-\Mark-\Refine steps. In this paper we show, however, that in many practical situations it is not necessary to use a fully resolved solution in the intermediate steps in order to obtain a good refinement pattern: numerical evidence shows that it is sufficient for the high frequencies of the error to be dumped, in order to identify the next grid in the sequence through the \Estimate-\Mark steps. In this context, the construction of a grid with excellent approximation properties may require as little as a single ascending step of a V-cycle multigrid method.

It is still not clear how to explain rigorously why the sequence of meshes constructed with S-AFEM is close to the one obtained by classical AFEM. In this work we examine the numerical behaviour of S-AFEM in a family of linear second order elliptic problems, and provide a qualitative analysis, based on classical results from the AFEM and multigrid literature, to point the reader towards currently open questions, and towards possible paths to complete the analysis.

In particular, we consider conforming discretizations of a Poisson problem, and of a class of drift-diffusion problems in 2D and 3D, and we show that
 \begin{itemize}
     \item[$\star$] 
the a-posteriori error estimator applied to the outcome of a single ascending phase of the V-cycle multigrid method triggers a \Mark step where the refinement pattern is substantially equivalent to the one that would be generated by a classical \Solve step, at a considerable fraction of the computational cost;
 \item[$\star$] even if the final grid is not exactly identical to the one that would be obtained with the classical AFEM, the accuracy of the final solution is comparable in most cases;
\item[$\star$] the S-AFEM algorithm is robust with respect to different smoothers, and with respect to different discretization degrees.
\end{itemize}

The article is organised as follows: we start by describing the general S-AFEM algorithm in Section~\ref{sec:s-afem}, where the main algorithm is exposed. The connection with the multilevel framework and with classical a posteriori error analysis is made in Section~\ref{sec:qualitative-analysis}, where a qualitative analysis of S-AFEM is presented. Section~\ref{numericalvalidation} is dedicated to the numerical validation of the algorithm, and presents a detailed campaign of simulations that shows when the S-AFEM algorithm can be used successfully. Finally, in Section~\ref{conclusions} we provide some conclusions and perspectives for future works.

\section{The S-AFEM algorithm}
\label{sec:s-afem}
We consider a class of linear elliptic, second-order, boundary value problems (BVPs), whose variational formulation reads: seek $u\in V$ s.t.  $\mathcal{A}u=f$ in $V$ under suitable boundary conditions, where $(V, \| \bullet \|)$ is a normed Hilbert space defined over a Lipschitz bounded domain $\Omega$, the linear operator $\mathcal{A}:V\rightarrow V^{\star}$ is a second order elliptic operator, and $f \in V^{\star}$ is a given datum. The finite element method provides numerical solutions to the above problem in a finite dimensional solution space $V_h \subset V$, typically made up by continuous and piecewise polynomial functions, and transforms the continuous problem above in a discrete model of type $\mathcal{A}_h u_h=f_h$ in $V_h$ under suitable boundary conditions, where, e.g., $\mathcal{A}_h = \mathcal{A}\mid_{V_h}.$ The overall procedure leads to the solution of a (potentially very large) linear algebraic system of equations of type $A \mathbf{u} = \mathbf{f}$ in $\mathbb{R}^N$, where $N=\dim(V_h)$.

Given an initial (coarse) triangulation $\mathcal{T}_1$, we consider a (a priori unknown) nested sequence of shape regular triangulations $\mathcal{T}_k$, for $k= 1, \dots, \bar{k}$, which induces a nested sequence of  finite element spaces
\begin{equation} \label{multilevel}
V_1 \subset V_2 \subset \dots \subset V_{\bar{k}} \equiv V_h,
\end{equation}
on which we define standard prolongation operators, considering the canonical embedding $ i_k^{k+1}: V_k \hookrightarrow V_{k+1}$ that embeds functions $u_k \in V_k$  in the space $V_{k+1}$. We denote by $I_{k}^{k+1}: \mathbb{R}^{N_{k}} \rightarrow \mathbb{R}^{N_{k+1}}$ the corresponding discrete matrices and we let $N_k:=  \dim (V_k),$ for $k=1, \dots, \bar{k}$. 

The sequence of grids and the solution on the finest grid is computed with the S-AFEM algorithm, defined as:

\begin{algorithm}\label{smoothedalgorithm}
Starting from an initial coarse mesh $\mathcal{T}_1,$ Solve $A_1 \mathbf{u}_1=\mathbf{f}_1$ in $\mathbb{R}^{N_1}$ to high accuracy and generate $\mathbf{u}_1$. Then, do steps $1.-4.$ for $k=2, \dots, \bar{k}-1$ or until criterion met
\begin{enumerate}
\item \Smooth: Compute $\ell$ smoothing iterations on the discrete system $A_k \mathbf{u}_k=\mathbf{f}_k$, with initial guess $\mathbf{u}^{(0)}_k:= I^k_{k-1} \mathbf{u}^{(\ell)}_{k-1}$, which produce $\mathbf{u}_k^{(\ell)} \in \mathbb{R}^{N_k}$ (take $\mathbf{u}^{(\ell)}_1= \mathbf{u}_1$).
\item \Estimate: Compute estimators $\eta_T(u^{\ell}_k)$ for all elements $T \in \mathcal{T}_k.$
\item \Mark: Choose a set of cells to refine $\mathcal{M}_k \subset \mathcal{T}_k$ based on $\eta_T(u^{\ell}_k).$
\item \Refine: Generate new mesh $\mathcal{T}_{k+1}$ by refinement of the cells in $\mathcal{M}_k.$
\end{enumerate}
Step $k=\bar{k}$: Solve the discrete system $A_{\bar{k}} \mathbf{u}_{\bar{k}}=\mathbf{f}_{\bar{k}}$ to the desired algebraic accuracy. \\
\emph{Output}: sequence of meshes $\mathcal{T}_{k}$, smoothed approximations $\mathbf{u}^{\ell}_{k}$, estimators $\eta(u^{\ell}_k)$, and final adapted-approximation $\mathbf{u}^{\ell}_{\bar{k}}$ such that $\| \mathbf{e}_{\bar{k}}\| \le tol$.
\end{algorithm}

The motivation behind the strategy at the base of S-AFEM is the numerical observation that classical residual-based a posteriori error estimators  \cite{verfurth1996review} used in the \Estimate step are mostly insensitive to low frequencies in the solution, as shown in Figure~\ref{fig:corner-2d-intro} for a benchmark example. 
   \begin{SCfigure}
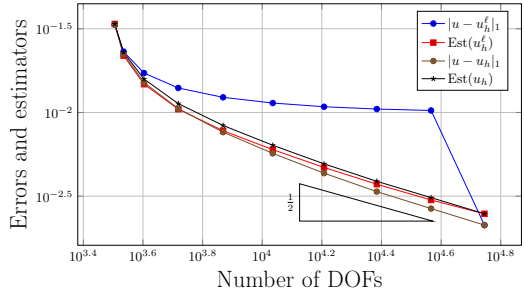
\resizebox{.55\linewidth}{!}{
\ErrorAndEstimatorSAFEMM{corner_2d_error}{
    \logLogSlopeTriangleReversed{0.8}{0.3}{0.1}{1/2}{black}{$\frac12$}; }}
  \caption{\footnotesize The values of the total error energy norm and of the error estimator for each loop of the classical AFEM $(| u-u_h|_1$ and $\text{Est}(u_h))$ and S-AFEM with $\ell=2$ smoothing iterations ($| u-u^{\ell}_h|_1$ and  $\text{Est}(u^{\ell}_h)$) for a classical Poisson problem on an L-shaped domain in 2D. The first and last loop are solved exactly by both methods.}
  \label{fig:corner-2d-intro}
\end{SCfigure} Their application to very inaccurate approximate solutions in the intermediate loops -- only capturing high frequency oscillations through a smoother -- produces an equally good grid refinement pattern at each step of the S-AFEM algorithm, with an accuracy on the final approximation step that is comparable to the one obtained with the classical AFEM algorithm, at a fraction of the computational cost.

In the S-AFEM algorithm, we capture the smoothest (i.e. less oscillatory) part of the discrete approximation in the first loop ($k=1$), by solving the discrete system exactly on the coarsest level. As the mesh is locally refined from one level to the other, we increase the higher portion of the spectrum of the matrix $A_k$. Thanks to the structure of the refinement in typical finite element methods, mostly high frequencies are added to the system, while low frequencies are substantially left unaltered.

Even though the distance between the algebraic approximation $u_h^\ell$ (coming from S-AFEM algorithm at step $k$) and the exact solution $u$ drifts away during the various steps of the algorithm (curve $|u-u_h^\ell|_1$ in Figure~\ref{fig:corner-2d-intro}), the error estimator evaluated on $u_h^\ell$ remains substantially attached to the error estimator evaluated on an algebraically exact solution $u_h$ computed on the same mesh (comparison between curves Est($u_h^\ell$) and Est($u_h$) in Figure~\ref{fig:corner-2d-intro}).

The first and last loops of the S-AFEM algorithm coincide with those of the classical AFEM. In intermediate S-AFEM loops, however, the solution $u_h^\ell$ is far from the exact algebraic solution $u_h$. These intermediate solutions serve
solely to the construction of the final grid, and find no other use in the final computations, therefore their inexactness is irrelevant, provided that the finally adapted grid provides a good approximation. Their role is \textit{instrumental} in triggering the
$\Estimate-\Mark-\Refine$ steps.

In our setting, intermediate steps are only computed through a fixed number of smoothing iterations, and have non-negligible algebraic errors. This is in contrast with the common practical assumption made in AFEM, where it is assumed that the \Solve step produces the \textit{exact} solution of the algebraic system. Recent developments dedicated a great deal of effort to account for inexactness of the algebraic approximations and introduce stopping criteria based on the interplay between discretization and algebraic computation in adaptive FEM. Among others, we mention the seminal contributions  \cite{becker1995adaptive, jiranek2010posteriori, arioli2013stopping, arioli2013interplay, papevz2018estimating, miracci2019multilevel,papevz2017residual,papevz2017sharp, daniel2020guaranteed, papevz2014distribution, MALLIK2020112367, daniel2020adaptive}. 

Nevertheless, most of this literature focuses on ways to estimate the algebraic error, without really exploiting the other side of the coin: inexact approximate solutions, with large algebraic error, offer large computational savings when used in the correct way. S-AFEM provides a good strategy to exploit this fact.

\section{Qualitative analysis of S-AFEM}
\label{sec:qualitative-analysis}
To fix the ideas, in this section we apply S-AFEM to a model Poisson problem with Richardson iteration as a smoother and discuss the interplay between the algebraic solution in intermediate steps, and the classical a-posteriori error estimator theory. 

A larger selection of problem types and smoother algorithms is tested in Section~\ref{numericalvalidation} where we compare the application of a fixed number of Richardson iterations or of the conjugate gradient (CG) method for symmetric systems, and of the generalized minimal residual method (GMRES) method for non-symmetric systems coming from the discretization of drift-diffusion problems.

Let $\Omega \subset \mathbb{R}^d$ $(d=1,2,3)$ be a bounded, polygonal domain (an open and connected set with polygonal boundary) with Lebesgue and Sobolev spaces $L^2(\Omega)$ and $H^1_0(\Omega)$. We look for the solution $u \in H^1_{0}(\Omega)$ such that
\begin{equation} \label{modelproblem}
- \Delta u  = f \,\,\, \textnormal{in} \,\,\, \Omega \,\,\, \textnormal{and}\,\,\,
 u= 0 \,\, \textnormal{on}\,\, \partial \Omega,
\end{equation}
where $f \in L^2(\Omega)$ is a given source term. We use standard notations for norms and scalar products in Sobolev spaces (cf.~\cite{adams2003sobolev}): for $u \in H^1_0 (\Omega)$ we write $ |u|_1 := (\int_{\Omega} | \nabla u|^2)^{1/2}$ and denote by $(\cdot, \cdot)$ the $L^2(\Omega)$- scalar product with corresponding norm $\| \cdot \|$. The weak form of \eqref{modelproblem} is to find $u \in H^1_{0}(\Omega)$ s.t. 
\begin{equation} \label{weakform}
(\nabla u, \nabla v)=(f, v) \,\,\forall v \in H^1_{0}(\Omega).
\end{equation}

We consider a shape regular family of triangulations $\{ \mathcal{T}_h \}_{h}$ of $\Omega$ in the sense of Ciarlet \cite{ciarlet2002finite}, depending on a parameter $h>0$ with shape regularity parameter $C_{\mathcal{T}_h}$ consisting of cells $T$ that are convex quadrilaterals in two dimensions, and convex hexahedrons in three dimensions. 

The set of all edges/faces $E$ of the cells is denoted by $\mathcal{E}_h$ and similarly, $\mathcal{E}_{h,int}:= \mathcal{E}_{h} \setminus \partial\Omega$ is the set of internal edges/faces. We use the Courant finite element space 
$V_h:= \{\varphi \in C^0(\Omega) \text{ s.t. } \varphi|_T \in P^k(T) \quad \forall T \in \mathcal T_h, \varphi = 0 \text{ on }  \partial \Omega\}\subset H^1_0(\Omega).$ The Galerkin solution $u_h \in V_h$ is obtained by solving the discrete system
\begin{equation} \label{discretesystem}
(\nabla u_h, \nabla v_h)=(f, v_h) \,\,\forall v_h \in V_h.
\end{equation}
In exact arithmetic, the discretization error $e_h:= u- u_h$ satisfies the standard orthogonality condition
\begin{equation} \label{galerkinorthogonality}
(\nabla (u - u_h), \nabla v_h)= 0\,\,\forall v_h \in V_h. 
\end{equation}

Let $N=  \dim (V_h),$ the discrete system \eqref{discretesystem} leads to a linear algebraic system of type 
\begin{equation}\label{linearsys}
    A \mathbf{u}= \mathbf{f}\,\,\textnormal{in}\,\,\mathbb{R}^N,
\end{equation}
where $A$ denotes the symmetric positive definite (SPD) stiffness matrix with entries $a_{ij}:=(\nabla \varphi_j, \nabla \varphi_i)$ $\forall \,i,j=1,.., N,$ $\mathbf{u}=[ u_1, \dots, u_N]^T$ denotes the coefficients vector in $\mathbb{R}^N$ of the discrete approximation $u_h = \sum_{j=1}^{N} u_j \varphi_j \in V_h$ and 
$\mathbf{f}= [ f_1, \dots, f_N]^T$ is the vector with entries  $f_j = (f, \varphi_j) \,\, \forall j = 1, .., N.$

In particular, we assume that an initial (coarse) triangulation $\mathcal{T}_1$ is given and we consider a (a priori unknown) nested sequence of shape regular triangulations $\mathcal{T}_k$, for $k= 1, \dots, \bar{k}$, which induces a nested sequence of finite element spaces $V_1 \subset V_2 \subset \dots \subset V_{\bar{k}}$. We let $N_k:=  \dim (V_k),$ for $k=1, \dots, \bar{k}$. By construction, the inequalities $N_1 < N_2 < \dots < N_{\bar{k}}$ hold true.
The associated discrete systems for each level  $k=1, 2, \dots, \bar{k}$ read
\begin{equation} \label{discretesystemk}
(\nabla u_k, \nabla v_k)=(f, v_k) \,\,\,\forall \, v_k\in V_k,
\end{equation}
and they generate linear systems of type 
\begin{equation} \label{linearsystemk}
A_k\mathbf{u}_k= \mathbf{f}_k
\end{equation}
of respective dimensions $N_k$.

The sequence of meshes in the classical AFEM algorithm is constructed through the \Solve-\Estimate-\Mark-\Refine steps, where the \Solve step should compute an algebraically exact solution of~\eqref{linearsystemk}. In the \Estimate step, standard residual-based a posteriori error estimators are the most widely used. They were first introduced in the context of FEM by Babu\v{s}ka and Rheinboldt in \cite{babuvska1978posteriori} and they have been thereafter widely studied in the literature (for a review, see~\cite{verfurth1996review, ainsworth2011posteriori}).

Their derivation is based on the residual functional associated to the Galerkin solution, which is defined as $\mathcal{R}\{u_h\} : H^1_0(\Omega) \longrightarrow \mathbb{R}$, $\mathcal{R}\{u_h\}:=  (f, \bullet ) - a(u_h, \bullet)$ with corresponding dual norm
\begin{equation}
\label{dualnormoftheresidual}
\|\mathcal{R}\{u_h\}\|_{\star} := \sup_{v \in H^1_0(\Omega) \setminus \{0\}} \frac{\mathcal{R}\{u_h\}(v)}{ |v|_1 } = \sup_{v \in H^1_0(\Omega) \setminus \{0\}} \frac{   (f,v)  - a(u_h, v)}{ |v|_1 } .
\end{equation}
The identity $|e_h|_1 = \| \mathcal{R}\{u_h\} \|_{\star}$ leads to reliable and efficient residual-based a posteriori bounds for the discretization error via estimation of the residual function, i.e., 
\begin{equation} \label{rel}
\|e_h\| \le C_{\text{rel}} \eta(u_h) + h.o.t._{\text{rel}}
\end{equation}
and 
\begin{equation} \label{eff}
\eta(u_h) \le C_{\text{eff}} \|e_h\|  + h.o.t._{\text{eff}},
\end{equation}
where the multiplicative constants  $C_{\text{rel}}$ and $C_{\text{eff}}$ are independent on the mesh size and h.o.t. denote oscillations of the right-hand side $f$, which are generally negligible w.r.t. $\|e_h\|$.

We use standard residual-based a posteriori error estimators which are locally defined through the jump of the gradient of the discrete approximation across the edges/faces $E$ of the cells, i.e., for a given function $v_h \in V_h,$ define  for $E \in \mathcal{E}_h$ and $T \in \mathcal{T}_h$
\begin{equation}
\begin{aligned}
\label{edge}
J_E(v_h)&:=h^{1/2}_E \left \| \left[ \frac{\partial v_h}{\partial n_E} \right] \right \|_E,\,\, J_T(v_h):= \sum_{E \in \partial T} J_E(v_h),\\
J(v_h)&:= \left( \sum_{E \in \mathcal{E}_h} J_E(v_h)^2 \right)^{1/2} = \left( \frac{1}{2} \sum_{T \in \mathcal{T}_h} J_T(v_h)^2 \right)^{1/2},
\end{aligned}
\end{equation}
where $\left[ \bullet \right]$ indicates the jump of a piecewise continuous function across the edge/face $E$ in normal direction $n_E$.
A classical upper bound on the discretization error using $J$ is given by the estimate
\begin{equation}
\label{upperboundestimate}
|u-u_h|_1 \le C^{\star} (\text{osc}^2 + J^2(u_h))^{1/2},
\end{equation}
where the constant $C^{\star}>0$ depends on the shape of the triangulation, on $\Omega$ and on $\Gamma$, but it is independent of $f$ and of the mesh-sizes $h_T$, and $\text{osc}$ is an oscillatory term (see~\cite{carstensen1999quasi} for the exact definition of $\text{osc}$ and for a proof of \eqref{upperboundestimate}).

In S-AFEM, we do not solve the linear systems \eqref{linearsystemk} for $k=2, \ldots, \bar k-1$, but only apply a smoother (typically, $\ell$ steps of a smoothing iteration) by taking as an initial guess the prolongation of the approximation from the previous level, obtaining an algebraically inexact approximation $u_h^\ell$ of $u_h$ in all intermediate steps. In this case, the total error in intermediate steps can be written as the sum of two contributions
\begin{equation} \label{totalerror}
\underbrace{u-u^\ell_h}_{\textnormal{total error}} = \underbrace{(u-u_h)}_{\textnormal{discretization error}} + \underbrace{(u_h-u^\ell_h)}_{\textnormal{algebraic error}}.
\end{equation}

A vast literature is dedicated to the extension of the standard residual-based a posteriori error estimator theory to incorporate in some way the algebraic error deriving from an inexact \Solve step. We refer to the seminal and investigative paper by Pape$\check{z}$ and Strako$\check{s}$ \cite{papevz2017residual} and the references therein for various approaches. 

In particular, in \cite{papevz2017residual} the authors give a detailed proof of a (worst case scenario) residual-based upper bound on the energy norm of the total error
\begin{equation} \label{bounddipapez}
    | u-v_h|^2_1 \le 2 C^2 ( J^2(v_h) + \text{osc}^2) + 2 C^2_{\text{intp}} |u_h-v_h|_1^2,
\end{equation}
for arbitrary $v_h \in V_h$. In order to provide sharper bounds, one should exploit the fact that $v_h$ is actually not arbitrary, but it originates from the \Smooth step of S-AFEM, i.e., it is the result of a smoothing iteration.

For simplicity of exposition, in this section we use a fixed number of Richardson iterations as a smoother, but other choices are possible (see, for example, the reviews in \cite{bramble2018multigrid, griebel1995abstract, xu1992iterative}). 

Given $\omega_k \in \mathbb{R}$ a fixed parameter and $\mathbf{u}_k^{(0)} \in \mathbb{R}^{N_k}$ an initial guess, a Richardson smoothing iteration for~\eqref{linearsystemk} takes the form
\begin{equation}\label{eq:Richardson_it}
\mathbf{u}_k^{(i+1)} = \mathbf{u}_k^{(i)} + \omega_k (\mathbf{f}_k - A_k\mathbf{u}_k^{(i)}) \,\,\, \textnormal{for} \,\,i=0,1, \dots, \ell
\end{equation}

After $i+1$ iterations, the error $\mathbf{e}_k^{(i+1)}:=\mathbf{u}_k - \mathbf{u}_{k}^{(i+1)}$ satisfies the error propagation formula $\mathbf{e}_k^{(i+1)} = M_k \mathbf{e}_k^{(i)}= \dots = M_k^{i+1} \mathbf{e}_k^{(0)}$, where $M_k:=Id_{N_k} - \omega_k A_k$ is the Richardson iteration matrix. When using Richardson method as an iterative solver, convergence takes place for $0<\omega_k< 2/\rho(A_k),$ where $\rho(A_k)$ is the spectral radius of $A_k$.
The optimal choice of the parameter $\omega_k$ is in this case $\omega_k = 1/\rho(A_k)$ (see, e.g.,~\cite{hackbusch1994iterative, saad2003iterative}).

The high frequency components of the error are reduced by a factor which is close to zero, while the low frequency components of the error are left substantially unaltered. The high frequencies are also the ones that have a greater influence on classical a-posteriori error estimators, justifying the S-AFEM algorithm. 

To show the effect of the smoothing iterations on the a-posteriori error estimator, we consider as an example case the Peak Problem in two dimensions as described in Subsection~\ref{subsec:2d_examples}, and we apply ten cycles of the classical AFEM algorithm using non-preconditioned Richardson iterations for the algebraic resolution of the system with initial guess given by the prolongation of the previous approximation for each cycle.

In Figure~\ref{Richardson-iteration-counts}
\begin{figure}
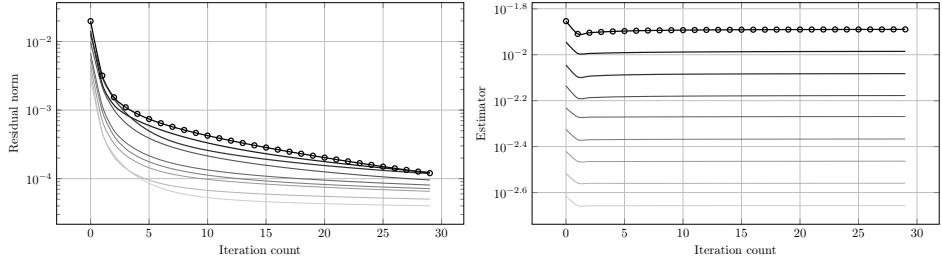
\resizebox{.98\textwidth}{!}{
  \AllResidualsVSestimators{Peak2DFullRichardson}{9}}
  \caption{Algebraic residual $\ell^2$-norm (left) and error estimator (right) for intermediate cycles of the classical AFEM algorithm when using Richardson iteration without preconditioner as a solver, with prolongation from the previous solution as starting guess. Darker lines correspond to earlier cycles. Only the first 30 iterations are shown.}
  \label{Richardson-iteration-counts}
\end{figure} we plot the $\ell^2$-norm of the residual $\mathbf{r}^{(\ell)}_k:= A_k \mathbf{e}_k^{(\ell)}$ and the value of the a-posteriori error estimator $J(u^{\ell}_k)$ for all cycles as the Richardson iteration count $\ell$ increases from $1$ to $30$.

The same behaviour is present in every refinement cycle: the first few Richardson iterations induce a rapid drop in the residual norm (due to convergence of the highly oscillatory terms in the solution), while the second part of the iterations converge very slowly, corresponding to the convergence speed of the low frequencies in the solution (typical of Richardson iteration). The estimator, on the other hand, stagnates after very few Richardson iterations (around two or three), suggesting that $J(u_h^{\ell})$ is almost the same as $J(u_h)$ for $\ell \ge 3$, i.e., the error estimator is mainly affected by the highly oscillatory components of the discrete algebraic solution $u_h^{\ell}$.

In this respect, classical results of a-posteriori error analysis do not provide sharp bounds for the estimator evaluated on the smoothed algebraic approximation $J(u_h^\ell)$, and this closeness remains an open problem.

One can combine the upper bound~\eqref{upperboundestimate} with a global lower bound~\cite{becker2009convergence} of the estimator evaluated on a generic $v_h$, i.e.,
\begin{equation}
\label{lowerboundestimate}
J^2(v_h) \le  C_{\star} ( |u-v_h|^2_1 + \text{osc}^2 )\,\, \forall v_h \in V_h,
\end{equation}
to relate $J(u_h^\ell)$ with $J(u_h)$. However, the result (proved in the following Lemma) remains a worst case estimate (similar to~\eqref{bounddipapez}), and fails to capture the behaviour that we observe, for example in Figure~\ref{fig:corner-2d-intro}, where $J(u_h)$ and $J(u_h^\ell)$ are substantially equivalent.

\begin{lemma}\label{theoremmain1}There exist positive constants $C_1, C_2, C_3$ that only depend on the minimum angle of the triangulation, on $\Omega$, and on $\Gamma$, that are independent of $f, u, u_h$, and of the mesh-sizes $h_T$ such that 
\begin{equation} \label{equationmain2}
J^2(v_h) \le C_1 J^2(u_h) + C_2 |u_h-v_h|^2_1 + C_3 \text{osc}^2 \,\,\forall v_h \in V_h.
\end{equation}
\end{lemma}
\begin{proof}
For a given function $v_h \in V_h,$ we decompose $u-v_h=(u-u_h) + (u_h-v_h)$ and we apply the equality $|u-v_h|^2_1 = |u-u_h|^2_1 + |u_h-v_h|^2_1$ (see, e.g. \cite{liesen2013krylov}) to the lower bound \eqref{lowerboundestimate}
\begin{equation}
\begin{aligned}
J^2(v_h) & \le  C_{\star} ( |u-v_h|^2_1 + \text{osc}^2 )\\
          & = C_{\star} ( |u-u_h|^2_1 + |u_h-v_h|^2_1 + \text{osc}^2 ) \\
          & \le C_{\star} ( {C^{\star}}^2 (\text{osc}^2 + J^2(u_h)) + |u_h-v_h|^2_1 + \text{osc}^2 ) \label{ketu}\\
          &= C_{\star} {C^{\star}}^2 J^2(u_h) + C_{\star}|u_h-v_h|^2_1 +  C_{\star}({C^{\star}}^2 +1) \text{osc}^2 \\
          &= C_1 J^2(u_h) + C_2 |u_h-v_h|^2_1 + C_3 \text{osc}^2,
\end{aligned}
\end{equation}
where we have used the upper bound \eqref{upperboundestimate} in \eqref{ketu}.
\end{proof}

If we apply Lemma \ref{theoremmain1} with $v_h = u^\ell_h$, we obtain an upper bound on $J^2(u^\ell_h)$ in terms of $J^2(u_h)$ and of the algebraic error. A similar result involving the full estimator $\eta(u^\ell_h)$ can be found in~\cite{arioli2013stopping}. 

It is still unclear how to improve the upper bound~\eqref{equationmain2} to explain why $J(u_h^\ell)$ and $J(u_h)$ are as close as the numerical evidence suggests. What remains to be proved is that a sharper estimate on the constant $C_2$ of~\eqref{equationmain2} may be obtained, that depends on the frequency content of $v_h$, showing that $C_2$ is small when $v_h = u_h^\ell$ stems from a smoothing iteration, bringing $J^2(u_h)$ close to $J^2(u_h^\ell)$ even though $|u_h^\ell-u_h|^2_1$ is in fact not small at all.

What can be done, however, is an estimate of the evolution of $|u_h^\ell-u_h|_1$ in the intermediate steps of S-AFEM, exploiting classical results of multigrid analysis. Next theorem provides such result when the smoothing iteration is performed using Richardson method.

\begin{theorem}[Algebraic Error propagation in S-AFEM] \label{teo1errore}
Let $\mathbf{e}^{(\ell)}_k := \mathbf{u}_{k}-\mathbf{u}^\ell_{k}$ denote the algebraic error  after $\ell$ smoothing iterations at step $k$ of S-AFEM for $k=2, \dots, \bar{k}-1$. Let $\mathbf{a}_{1} = 0$ and define
\begin{equation}\label{ai}
    \mathbf{a}_{k+1}:= \mathbf{u}_{k+1} - I_k^{k+1}\mathbf{u}_k \in \mathbb{R}^{N_{k+1}},\qquad k=1,\dots, \bar{k}-1
\end{equation}
denote the difference between the  exact algebraic solution $\mathbf{u}_{k+1}$ at level $k+1$ and the prolongation to level $k+1$ of the exact algebraic solution $\mathbf{u}_k$ at level $k$. Then, the following error propagation formula holds true\\
\begin{align}
    &\mathbf{e}^{(\ell)}_{k+1}= {M_{k+1}}^{\ell} (\mathbf{a}_{k+1} +   I_k^{k+1} \mathbf{e}_k^{(\ell)}),\,\,\text{ for }\,\,k=1, \dots,\bar{k}-1.
\end{align}

\end{theorem}

\begin{proof}
In the \Prolongate step of S-AFEM, the outcome $\mathbf{u}_k^{(\ell)}$ of the \Smooth procedure at step $k$ is prolongated to step $k+1$, and used as an initial guess in Richardson iteration at step $k+1$. We can write
\begin{equation}
\begin{aligned}
\mathbf{u}_{k+1}^{(0)} &= I_k^{k+1}\mathbf{u}_k^{(\ell)} \\
                        &= I_k^{k+1} \mathbf{u}_k - I_k^{k+1} \mathbf{e}_k^{(\ell)},  
\end{aligned}
\end{equation}
and therefore express the initial error at step $k+1$ as
\begin{equation}
\begin{aligned}
 \mathbf{e}_{k+1}^{(0)}&= \mathbf{u}_{k+1} - \mathbf{u}_{k+1}^{(0)}\\
 &= \mathbf{u}_{k+1} - I_{k}^{k+1} \mathbf{u}_{k} + I_{k}^{k+1} \mathbf{e}_{k}^{(\ell)}\\
 &= \mathbf{a}_{k+1} + I_{k}^{k+1} \mathbf{e}_{k}^{(\ell)}. 
\end{aligned}
\end{equation}

By applying directly the property of the error propagation formula for Richardson iterations, we obtain the final algebraic error at level $k+1$:
\begin{equation}
\begin{aligned}
\mathbf{e}_{k+1}^{(\ell)} &= {M_{k+1}}^{\ell}    \mathbf{e}_{k+1}^{(0)}  \\
                       &= {M_{k+1}}^{\ell} ( \mathbf{a}_{k+1} + I_k^{k+1} \label{kjo} \mathbf{e}_k^{(\ell)}),
\end{aligned}
\end{equation}
which proves the recursive formula.
\end{proof}
 
Theorem~\ref{teo1errore} shows that the nature of the algebraic error in S-AFEM is the result of $\ell$ smoothing iterations applied to a vector that accumulates the (smoothed) prolongation of the exact algebraic solution coming from step zero. 

The rationale behind S-AFEM is then that in the first step we perform a full \Solve, resulting in a negligible algebraic error, and the only components of the error that we are introducing when prolongating from step $k$ to step $k+1$ are high frequency errors (introduced by local refinement). These, however, are reduced very quickly by $\ell$ steps of smoothing iterations.

The residual algebraic error that persists as S-AFEM proceeds (clearly visible in Figure~\ref{fig:corner-2d-intro}) seems not to have a detrimental effect on $J(u_h^\ell)$. Such an algebraic error is probably confined on medium to low frequencies, and shows no noticeable effect on the \Estimate and \Mark steps of S-AFEM.

Although the value we plot in Figure~\ref{Richardson-iteration-counts} for the estimator is a global one, and gives no information on the distribution of the local estimator on the grid, it is a good hint that the overall behaviour of such distribution will not be changing too much after the first few Richardson iterations. We show some numerical evidence  that this is actually the case in the numerical validation provided in Section \ref{numericalvalidation}.

To summarize the idea behind S-AFEM, we argue that in the intermediate AFEM cycles it is not necessary to solve exactly the discrete system. What matters instead is to capture accurately the highly oscillatory components of the discrete approximation.  Low frequency components \textit{may} have an influence on the error estimator, however, this is mostly a \textit{global} influence, that has a small effect on the cells that will actually be marked for refinement in the \Mark step. As an example, consider Figure~\ref{fig:peak-3d-error}, where a Peak problem in 3D is solved using both AFEM and S-AFEM. The estimator evaluated on $u_h^\ell$ in this case is farther away from the one evaluated on $u_h$ w.r.t. the same problem in two dimensions, but the convergence rate of the solution obtained with the sequence of grids constructed with S-AFEM is still the optimal one, and comparable to the one obtained with AFEM at a fraction of the computational cost.

\section{Numerical validation}
\label{numericalvalidation}
We consider a class of drift-diffusion problems of the following form
\begin{equation}\label{eq:diffusionetrasportogenerale}
  \begin{aligned}
    & -\Delta u + \boldsymbol{\beta} \cdot \nabla u = f &&\textnormal{in}\,\,\Omega\\
    & u = u_g &&\textnormal{on}\,\,\partial\Omega,
  \end{aligned}  
\end{equation}
in two and three dimensions. 

We compare the classical AFEM algorithm with S-AFEM based on three different smoothing strategies: Richardson iterations (Richardson), conjugate gradient iterations (CG), and generalized minimal residual iterations (GMRES) in the symmetric case ($\boldsymbol{\beta}=0$), and GMRES alone in the non-symmetric case ($\boldsymbol{\beta}\neq0$), and for different discretization degrees.

We test two classical experiments used to benchmark adaptive finite element methods when $\boldsymbol{\beta}$ is zero, and a simple drift-diffusion problem with constant transport term $\boldsymbol{\beta}$ to test S-AFEM in general drift-diffusion problems.

The numerical results presented in this paper were realized using a custom \texttt{C++} code based on the \dealii library~\cite{dealii9.0,ArndtBangerthClevenger-2019-a,Daniel-Arndt-2020-a}, and  on the \texttt{deal2lkit} library~\cite{sartori18}. In the \Mark step, we use the classical D\"orfler marking strategy~\cite{dorfler1996convergent}: for any level $k$ we mark for refinement the subset of elements
\begin{equation}
    \mathcal{M}_k:=  \{T \in \mathcal{T}_k : \eta_T \ge L \},
\end{equation}
where $L$ is a  \textit{threshold error},  defined as the largest value such that
\begin{equation}
    \Theta \sum_{T \in \mathcal{T}_k} \eta^2_{T} \le \sum_{T \in \mathcal{M}_k} \eta^2_{T}.
\end{equation}
The parameter $\Theta$ is such that $0 \le \Theta \le 1$, where $\Theta= 1$ corresponds to an almost uniform refinement, while $\Theta=0$ corresponds to no refinement. In our numerical tests, unless otherwise stated, we set $\Theta=0.3$. The refinement strategy that we adopt in this work is based on the use of hanging nodes~\cite{dealii2007}.

The results presented in this section are a subset of the full campaign of simulations presented in \cite{mulita2019}.

\subsection{Two-dimensional examples: pure diffusion, bi-linear case, Richardson smoother}
~\label{subsec:2d_examples}
\paragraph{Smooth domain, peak right hand side}
The first example we consider consists in solving the model problem with no transport term ($\boldsymbol{\beta}=0$) on a square domain, with a custom forcing term that contains a peak in a specified point in the domain, forcing the exact solution to be
\begin{equation}
  \label{eq:peak-solution-2d}
  u(x,y) = x(x-1)y(y-1)e^{-100\big( (x-0.5)^2+(y-0.117)^2\big)},
\end{equation} as shown in Figure~\ref{fig:peak-2d-solution}.
\begin{figure}
  \centering
  \includegraphics[width=.80\linewidth]{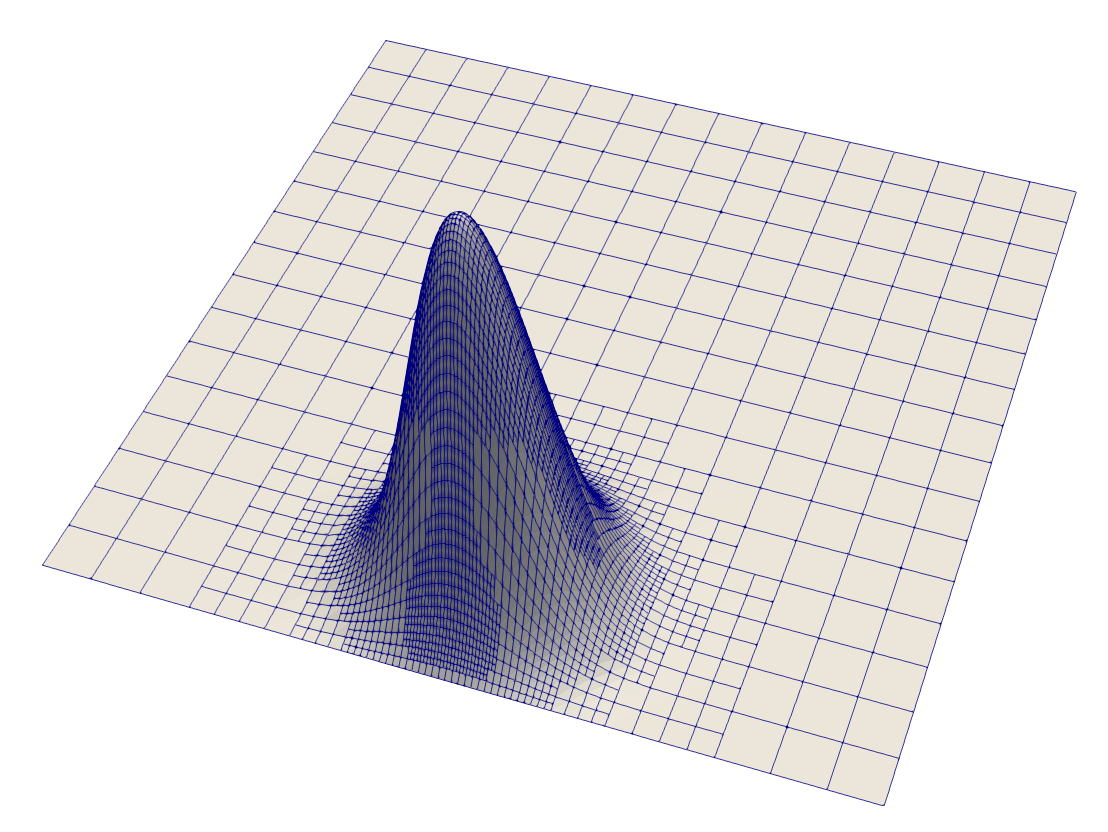}
  \caption{Solution to the Peak Problem in 2D~\eqref{eq:peak-solution-2d}.}
  \label{fig:peak-2d-solution}
\end{figure}

\paragraph{L-shaped domain, smooth right hand side}
In the second two-dimensional test case in pure diffusion, we consider a L-shaped domain, i.e., a square where the upper right corner is removed, and the reentrant corner coincides with the origin. No forcing term is added to the problem, but the boundary conditions are set so that the following exact solution is obtained (when expressed in polar coordinates)
\begin{equation}
  \label{eq:fichera-corner-2d}
  u(r,\theta) = r^{2/3} \sin\left(\frac{2\theta+5\pi}{3}\right),
\end{equation}
as shown in Figure~\ref{fig:corner-2d-solution}.
\begin{figure}
  \centering
  \includegraphics[width=.80\linewidth]{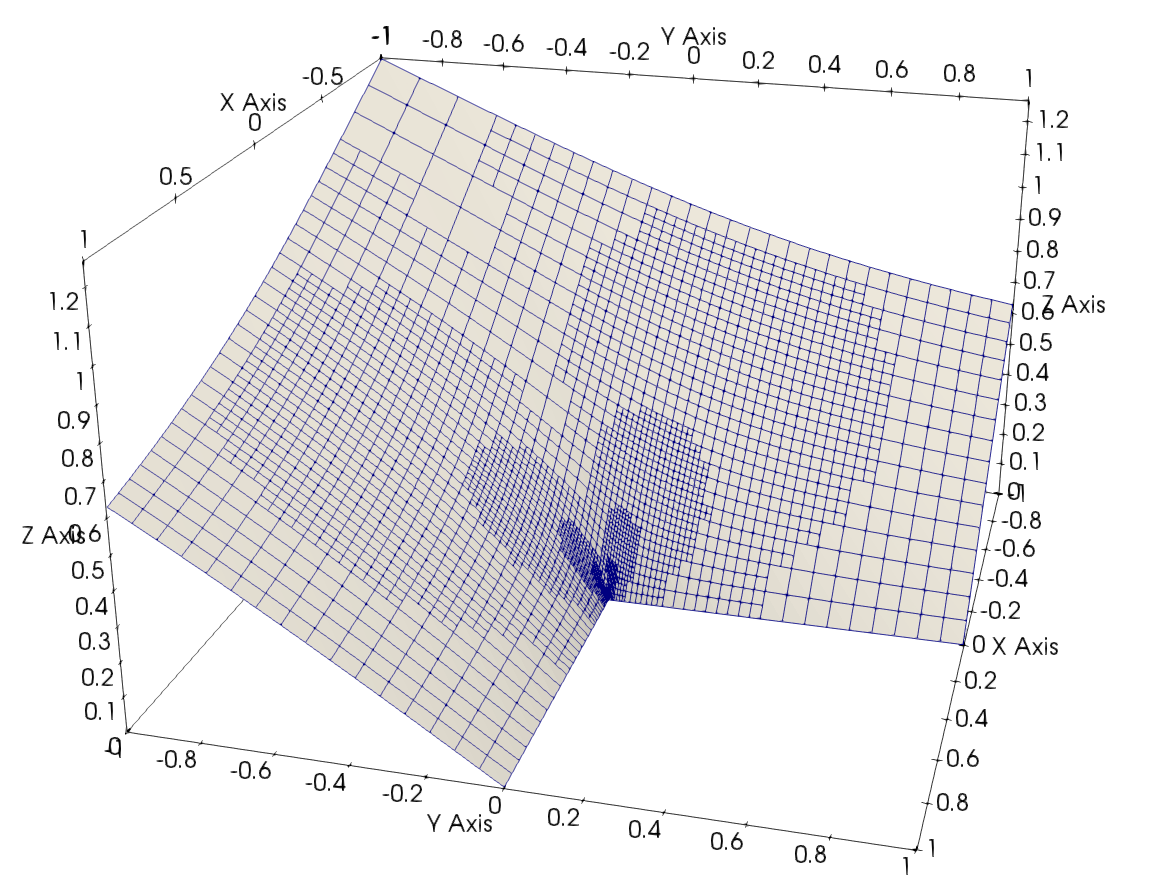}
  \caption{Solution to the L-shaped domain Problem in 2D~\eqref{eq:fichera-corner-2d}.} 
  \label{fig:corner-2d-solution}
\end{figure}

In both cases, we apply ten cycles of classical AFEM and of S-AFEM, respectively. For the AFEM algorithm, we use the CG method as iterative solver, with an algebraic multigrid preconditioner (AMG), and we iterate until the $\ell^2$-norm of the residual is below a tolerance of $10^{-12}$ for each cycle. For S-AFEM, we modify the intermediate cycles and we only apply three Richardson iterations. For reference, we report a comparison between the cells marked for refinement by AFEM and S-AFEM after four cycles for the Peak Problem in Figure \ref{fig:comparisonmarkpeak2d} and after nine cycles for the L-shaped domain Problem in Figure \ref{fig:comparisonmarkcorner_2d}. 

In both cases, the set of marked cells, although different in some areas, produces a refined grid that is very similar between the classical AFEM and the S-AFEM, and where the accuracy of the final solution is comparable. 

\begin{figure} 
  \centering
  \includegraphics[width=.49\linewidth]{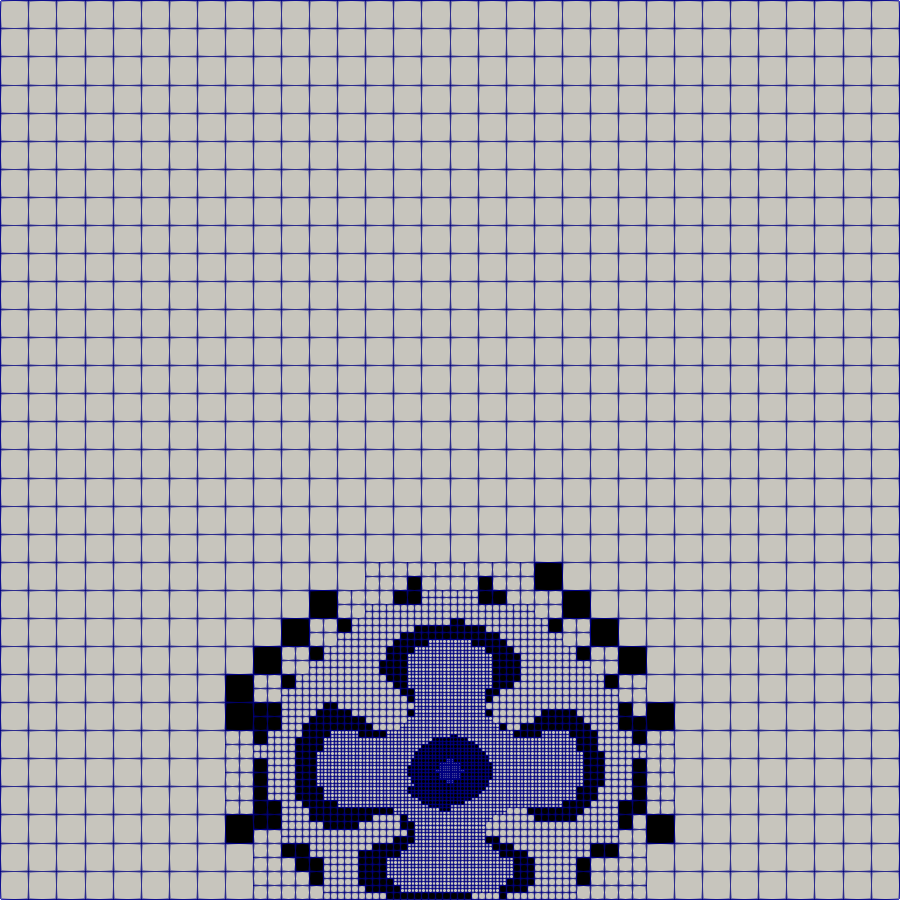}
  \hfill
  \includegraphics[width=.49\linewidth]{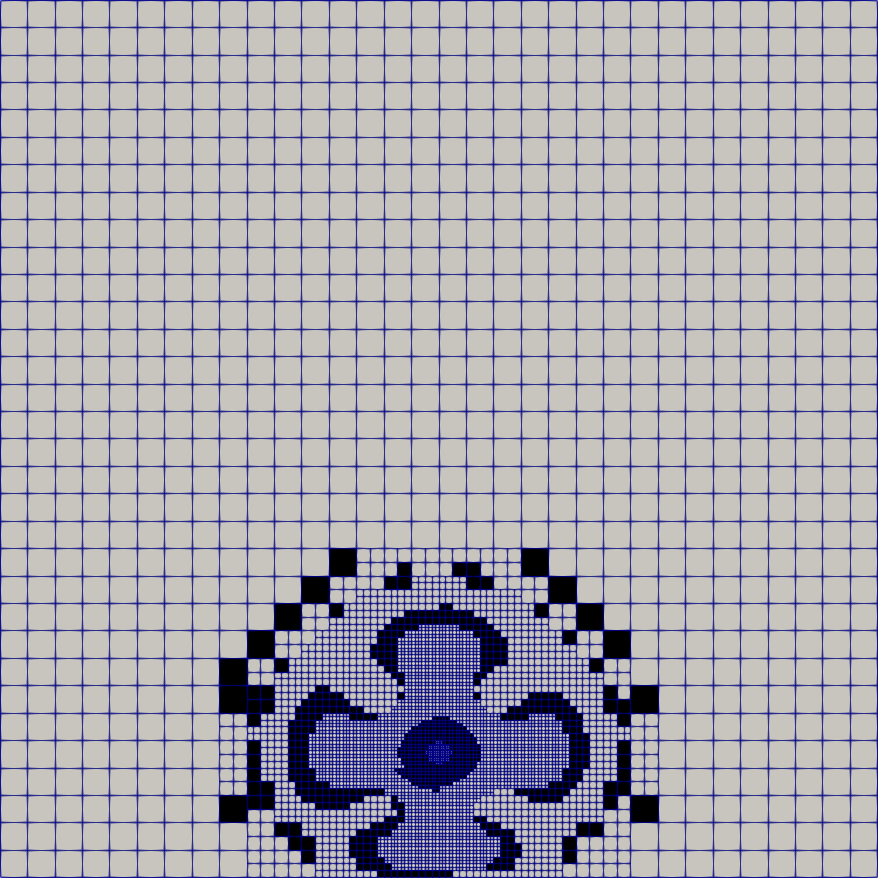}
  \caption{Comparison between the cells marked for refinement in AFEM (left) and S-AFEM (right) after $9$ cycles for the Peak Problem in 2D.}
  \label{fig:comparisonmarkpeak2d}
\end{figure}

\begin{figure}
  \centering
  \includegraphics[width=.49\textwidth]{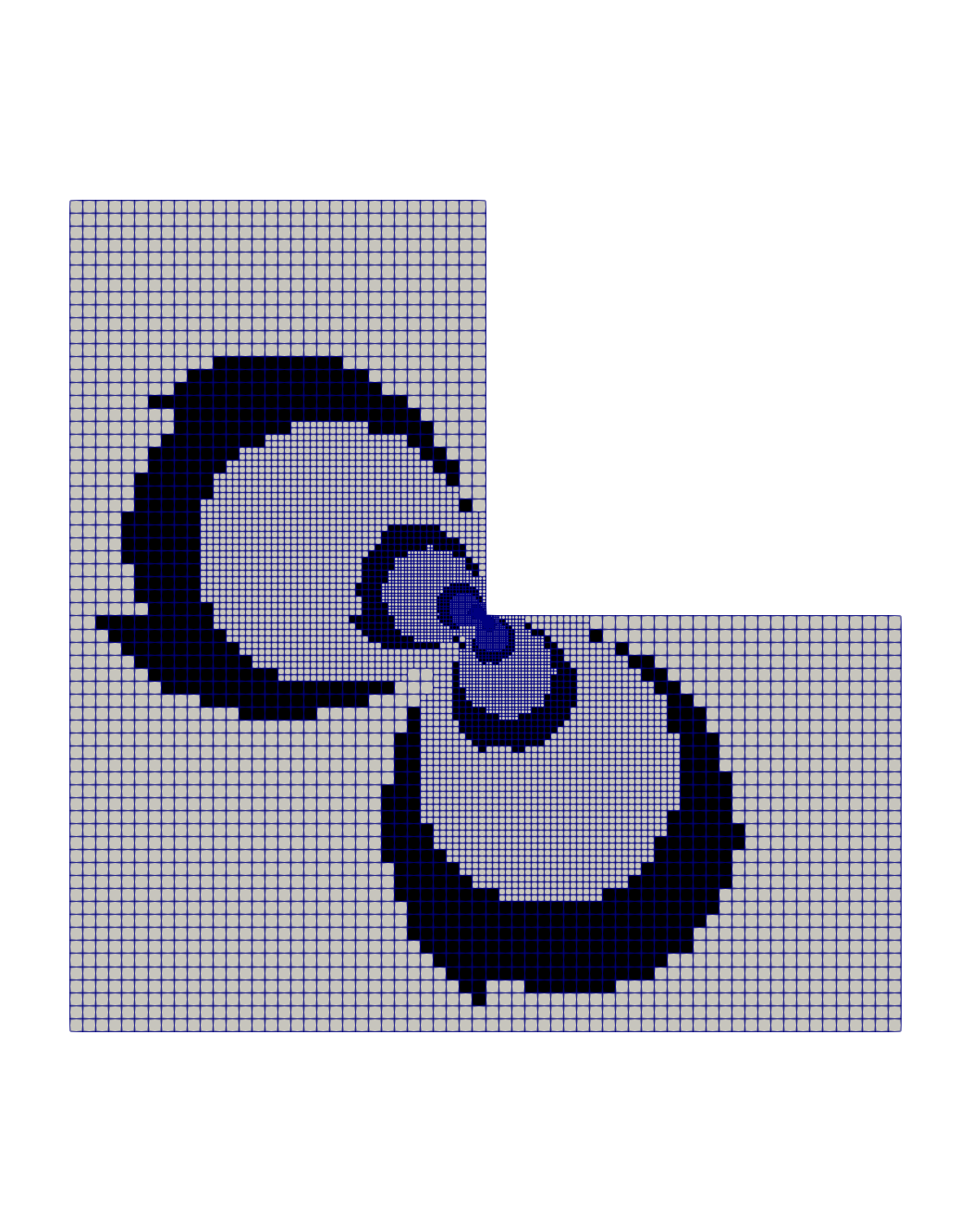}
  \hfill
  \includegraphics[width=.49\textwidth]{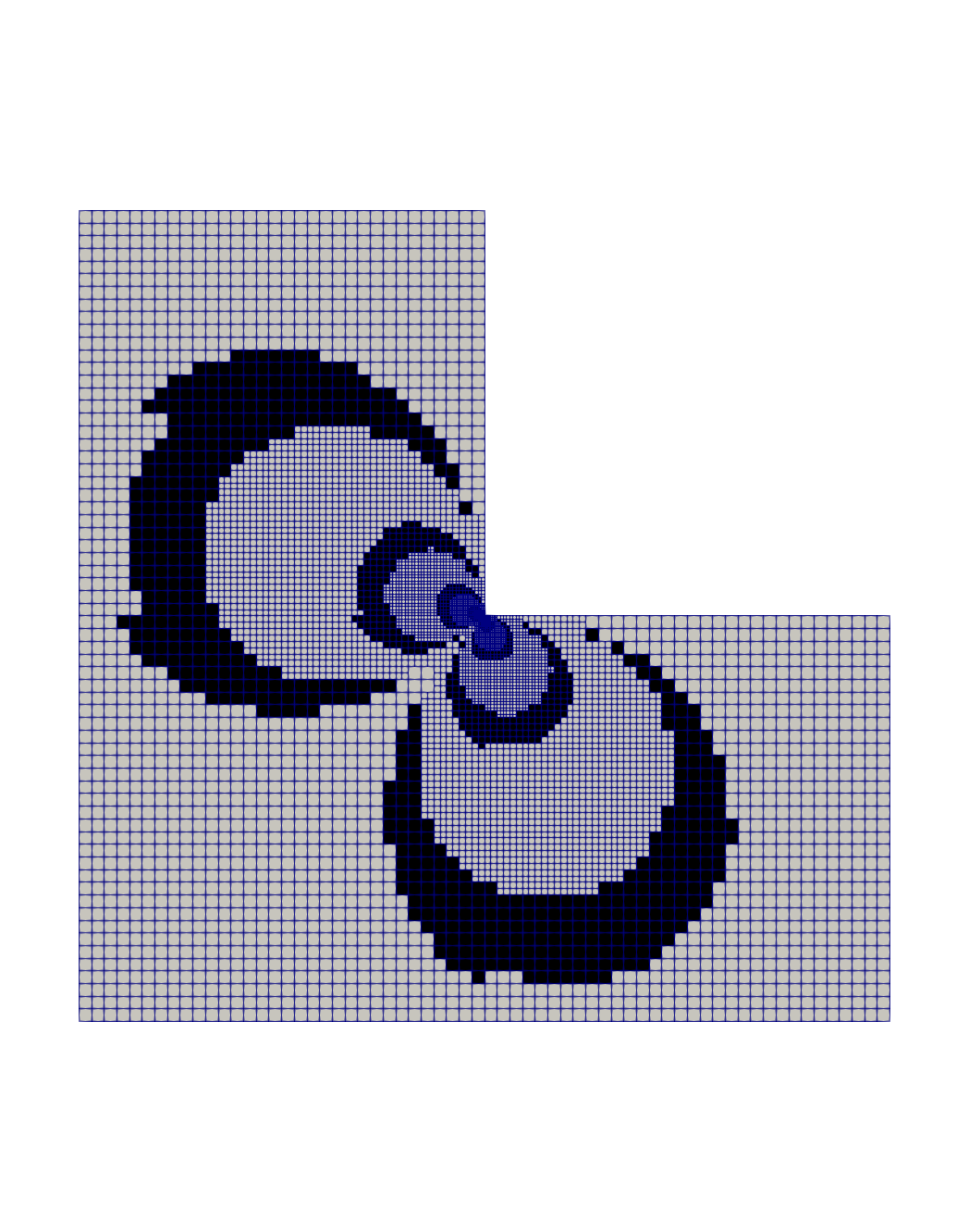}
  \caption{Comparison between the cells marked for refinement in AFEM (left) and S-AFEM (right) after $5$ cycles for the L-shaped domain Problem in 2D.}
  \label{fig:comparisonmarkcorner_2d}
\end{figure}

In Figures~\ref{fig:peak-2d-error} and~\ref{fig:corner-2d-error}  \begin{figure}\centering\resizebox{.65\linewidth}{!}{
  \ErrorAndEstimatorSAFEM{peak_2d_error}{
    \logLogSlopeTriangleReversed{0.82}{0.3}{0.08}{.5}{black}{$\frac12$}; }}
  \caption{Values of the total error in the $H^1$ semi-norm and of the error estimator for each loop of the classical AFEM $(| u-u_h|_1$ and $J(u_h))$ and S-AFEM with $\ell=3$ smoothing iterations ($| u-u^{\ell}_h|_1$ and $J(u^{\ell}_h)$) for the Peak Problem in 2D.}
  \label{fig:peak-2d-error}
\end{figure}
   \begin{figure}\centering\resizebox{.65\linewidth}{!}{
   \ErrorAndEstimator{peak_2d_afem_error}{
     \logLogSlopeTriangle{0.92}{0.3}{0.4}{1/2}{black}{$\frac12$}; }}
  \caption{Values of the total error $|u-u_h|_1$ and error estimator $J(u_h)$ for the Peak Problem in 2D, using classical AFEM.}
  \label{fig:peak-2d-error-afem}
\end{figure} we compare the values of the global estimators $J(u_h)$ and $J(u_h^{\ell})$ and of the $H_1$ semi-norm of the total errors for each cycle for the Peak Problem, and for the L-shaped domain Problem respectively, when using S-AFEM. For reference, Figures~\ref{fig:peak-2d-error-afem} and~\ref{fig:corner-2d-error-afem} show the error and the estimator in the classical AFEM algorithm for the two examples. Notice that the first step of AFEM and of S-AFEM are the same. The last step in the S-AFEM case shows comparable results as in the AFEM algorithm for both examples.  

Notice that in S-AFEM the value of the global estimator is almost the same of the one that would be obtained by solving using CG preconditined with AMG ($J(u_h)$ in Figures~\ref{fig:peak-2d-error} and~\ref{fig:corner-2d-error}), showing that in the two dimensional setting the error estimator~\eqref{edge} is mainly affected by the high frequencies of the discrete solution, which are well captured with just a few Richardson iterations.  On the other hand, the total error increases in the intermediate cycles, due to the algebraic error that has been accumulated by applying smoothing iterations instead of solving the algebraic problem until convergence, as quantified by Theorem~\ref{teo1errore}. This error measures the distance between the exact algebraic solution and the smooth non-oscillatory components of the approximate solution that are not captured by Richardson iteration, and have little or no influence on the error estimator, and therefore on the generated grid. After ten cycles, we solve the algebraic problem until converge using CG and AMG, as in the first cycle, and we obtain a solution whose error is controlled  by the estimator, as expected.

\begin{figure}\centering\resizebox{.70\linewidth}{!}{
  \ErrorAndEstimatorSAFEM{corner_2d_error}{
    \logLogSlopeTriangleReversed{0.8}{0.3}{0.1}{1/2}{black}{$\frac12$}; }}
  \caption{Values of the total error $H^1$ semi-norm and of the error estimator for each loop of the classical AFEM $(| u-u_h|_1$ and $J(u_h))$ and S-AFEM with $\ell=3$ smoothing iterations ($| u-u^{\ell}_h|_1$ and $J(u^{\ell}_h)$) for the L-shaped domain Problem in 2D.}
  \label{fig:corner-2d-error}
\end{figure}

\begin{figure}
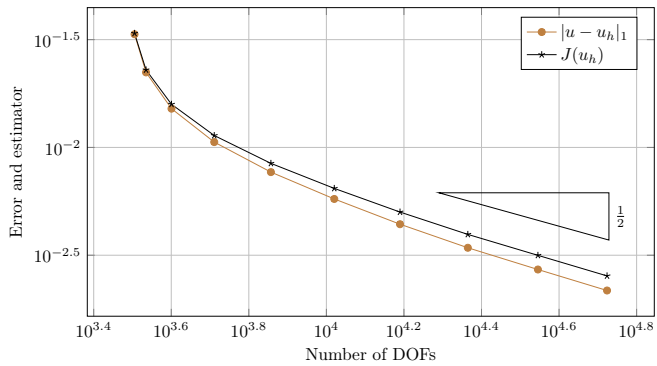
\centering\resizebox{.70\linewidth}{!}{
   \ErrorAndEstimator{corner_2d_afem_error}{
     \logLogSlopeTriangle{0.92}{0.3}{0.4}{1/2}{black}{$\frac12$}; }}
  \caption{Values of the total error $|u-u_h|_1$ and error estimator $J(u_h)$ for the L-shaped domain Problem in 2D, using classical AFEM.}
  \label{fig:corner-2d-error-afem}
\end{figure}

\subsection{Three-dimensional examples: pure diffusion, bi-linear case, Richardson smoother}

\paragraph{Smooth domain, peak right hand side}
The first three-dimensional test case that we propose is a model problem on a cube domain, where the forcing term contains a peak in a specified point that forces the exact solution to be given by
\begin{equation}
  \label{eq:peak-solution-3d}
  u(x,y,z) = x(x-1)y(y-1)z(z-1)e^{-100\big( (x-0.5)^2+(y-0.117)^2+(z-0.331)^2\big)}
\end{equation}
as shown in Figure~\ref{fig:peak-3d-solution}.
\begin{figure}
  \centering
  \includegraphics[width=.60\linewidth]{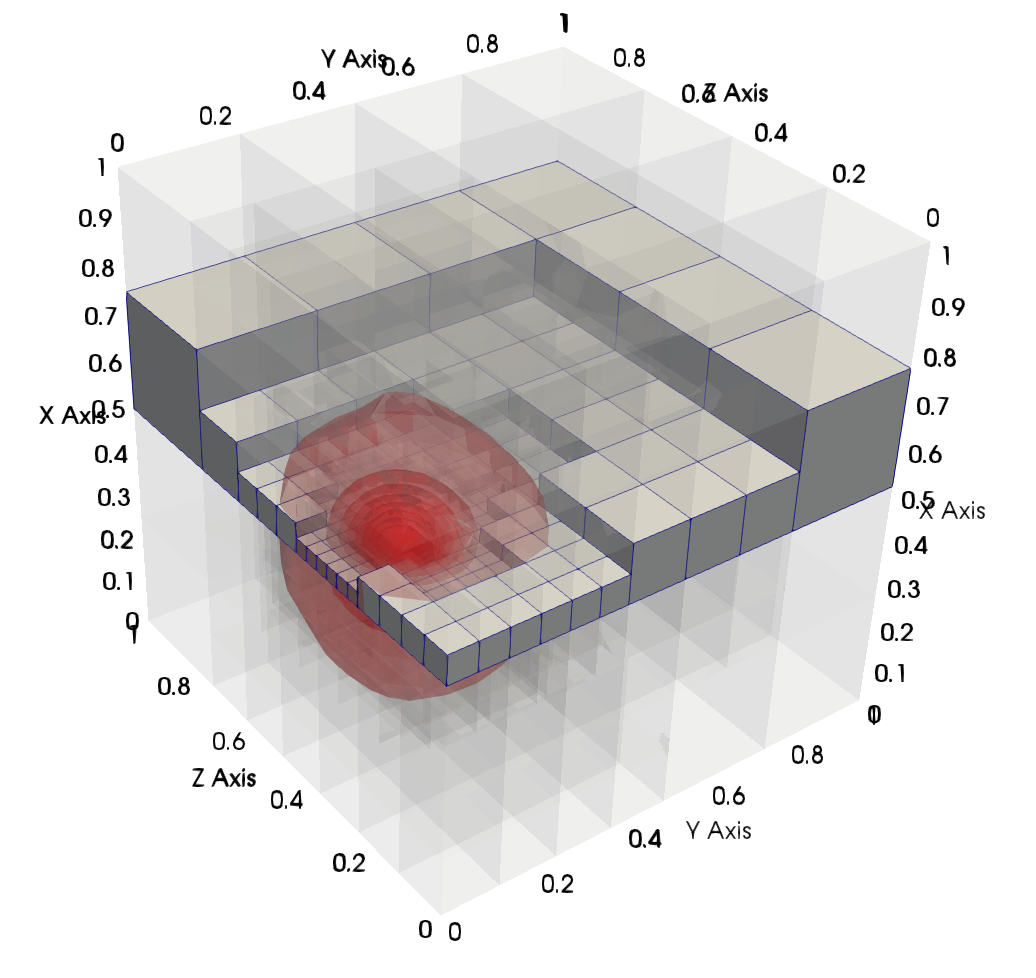}
  \caption{Solution to the Peak Problem~\eqref{eq:peak-solution-3d} in 3D.}
  \label{fig:peak-3d-solution}
\end{figure}

\paragraph{Fichera corner domain, smooth right hand side}
In the second three-dimensional example, we consider the classic Fichera corner domain, i.e., a cube where the upper right corner is removed, and the reentrant corner coincides with the origin. We set the exact solution to be 
\begin{equation}
  \label{eq:fichera-corner-3d}
  u(r,\theta, \phi) = r^{1/2},
\end{equation}
and we add a forcing term that induces the above exact solution
\begin{figure}
  \centering
  \includegraphics[width=.60\linewidth]{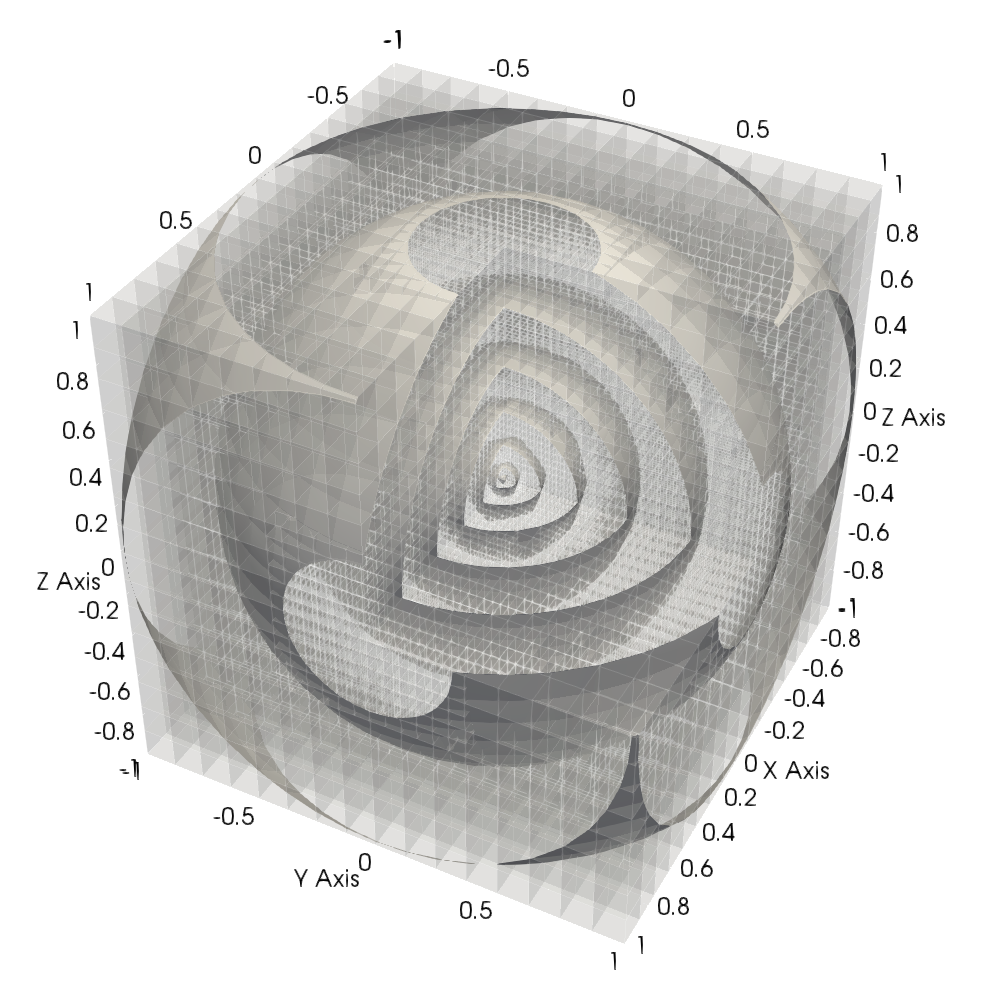}
  \caption{Solution to the Fichera domain Problem~\eqref{eq:fichera-corner-3d} in 3D.}
  \label{fig:corner-3d-solution}
\end{figure}
as shown in Figure~\ref{fig:corner-3d-solution}.

In both examples, the estimator applied to the algebraic solution after three smoothing steps (see Figures~\ref{fig:peak-3d-error} \begin{figure}
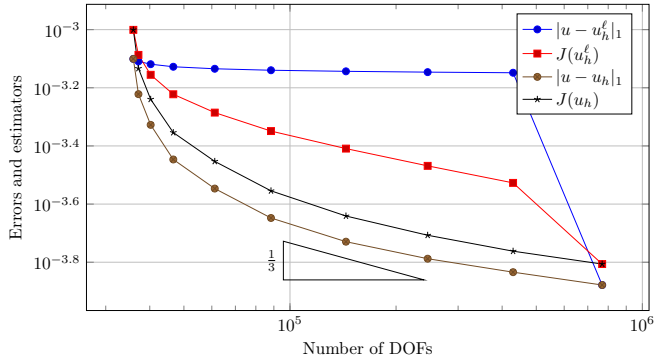

    \centering\resizebox{.70\linewidth}{!}{
   \ErrorAndEstimatorSAFEM{peak_3d_error}{
  \logLogSlopeTriangleReversed{0.6}{0.25}{0.1}{1/3}{black}{$\frac13$}; }}
  \caption{Values of the total error $H^1$ semi-norm and of the error estimator for each loop of the classical AFEM $(| u-u_h|_1$ and $J(u_h))$ and S-AFEM with $\ell=3$ smoothing iterations ($| u-u^{\ell}_h|_1$ and $J(u^{\ell}_h)$) for the Peak Problem in 3D.}
  \label{fig:peak-3d-error}
\end{figure} \begin{figure}
   \centering\resizebox{.70\linewidth}{!}{
   \ErrorAndEstimator{peak_3d_afem_error}{
  \logLogSlopeTriangle{0.92}{0.3}{0.4}{1/3}{black}{$\frac13$}; }}
  \caption{Values of the total error $|u-u_h|_1$ and error estimator $J(u_h)$ for the Peak Problem in 3D, using classical AFEM.}
  \label{fig:peak-3d-error-afem}
\end{figure} and~\ref{fig:corner-3d-error}) \begin{figure}\centering\resizebox{.70\linewidth}{!}{
   \ErrorAndEstimatorSAFEM{corner_3d_error}{
     \logLogSlopeTriangleReversed{0.7}{0.3}{0.1}{1/3}{black}{$\frac13$}; }}
  \caption{Values of the total error $H^1$ semi-norm and of the error estimator for each loop of the classical AFEM $(| u-u_h|_1$ and $J(u_h))$ and S-AFEM with $\ell=3$ smoothing iterations ($| u-u^{\ell}_h|_1$ and $J(u^{\ell}_h)$) for the Fichera corner domain Problem in 3D.}
  \label{fig:corner-3d-error}
\end{figure}  \begin{figure}
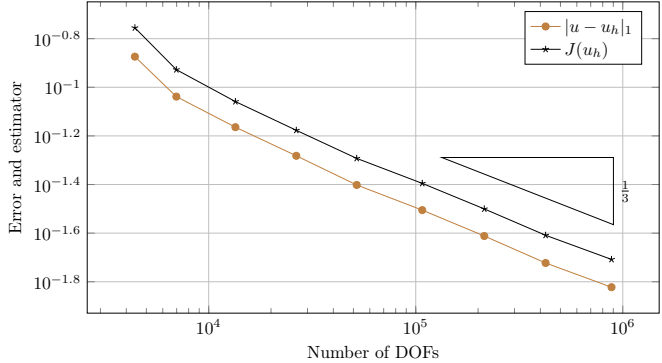
\centering\resizebox{.70\linewidth}{!}{
   \ErrorAndEstimator{corner_3d_afem_error}{
     \logLogSlopeTriangle{0.92}{0.3}{0.5}{1/3}{black}{$\frac13$}; }}
  \caption{Values of the total error $|u-u_h|_1$ and error estimator $J(u_h)$ for the Fichera corner Problem in 3D, using classical AFEM.}
  \label{fig:corner-3d-error-afem}
\end{figure} seems to be more sensitive to the low frequency content of $u_h^{\ell}$. 

For reference, Figures~\ref{fig:peak-3d-error-afem} and~\ref{fig:corner-3d-error-afem} show the error and the estimator in the classical AFEM algorithm for the two examples. In the three-dimensional case Lemma~\ref{theoremmain1} and Theorem~\ref{teo1errore} provide a sharper estimate, and we don't observe the same behaviour as in the two-dimensional case (i.e., $J(u_h^\ell)$ does not seem to remain close to $J(u_h)$). Nonetheless, the difference in accuracy at the final step between AFEM and S-AFEM is negligible also in this case, showing that the  differences in the refinement patterns between AFEM and S-AFEM remain small and do not hinder the final accuracy. 

\subsection{Robustness with respect to approximation degree, smoothing algorithms, and number of smoothing steps}

We now consider different variants of our S-AFEM algorithm, where we apply a different number of smoothing iterations, and different smoother types in the intermediate steps (respectively,  Richardson iteration, CG, and GMRES smoothers), for high order finite element discretizations of the pure diffusion case. 

We apply both AFEM and S-AFEM to the two-dimensional Corner Problem~\eqref{eq:fichera-corner-2d} and we show a comparison for different fixed FEM degrees, as $deg=1,2, 3$, and for different choices of smoothers for the intermediate cycles, respectively Richardson  iteration, the CG method, and the GMRES method. For all cases, we plot the value of the error estimator $J$ and the value of the $| \cdot|_1$ semi-norm of the total error, as the number of smoothing iterations $\ell$ increases from $1$ to $5$ in Figures~\ref{fig:corner_2d_deg1}-\ref{fig:corner_2d_deg3_gmres}.

For bi-linear finite elements (cf.~Figures~\ref{fig:corner_2d_deg1}--\ref{fig:corner_2d_deg1_gmres}), all considered smoothers (Richardson, CG, and GMRES) turned out to be good. In all cases, the estimator $J(u_h^{\ell})$  with $\ell=1, 3, 5$ exhibits the same behaviour (same order of convergence) of the estimator $J(u_h)$, showing that one or two smoothing iterations would be enough for the intermediate cycles. When we look at the total error, the CG behaves better leading to less error accumulated at the intermediate levels as shown in Figure~\ref{fig:corner_2d_deg1_cg}, while Richardson behaves the worst (cf.~Figure~\ref{fig:corner_2d_deg1}). Nevertheless, in all cases the accuracy of the final approximation for the last cycle obtained by S-AFEM, is almost the same to the one that is generated by classical AFEM.

As the polynomial degree increases, we observed that the D\"orfler marking strategy does not provide good refinement patterns for the different problems, unless a fine tuning is made on the marking parameter. Using the same value for $\Theta$ used for degree one, no cells are marked for refinement in higher order finite elements, making the choice for this parameter too much problem dependent and polynomial degree dependent. As an alternative marking strategy, we opted for a marking criterion where a fraction of $1/3$ of the cells with the largest error indicators are selected for refinement, leading to an increase of the number of degrees of freedom of roughly a factor of two in each refinement cycle, independently on the problem type. 

For $deg=2$, all smoothers work well exhibiting a quasi optimal convergence order compared to classical AFEM, and the accuracy of the final approximation is almost the same, as shown in 
Figures~\ref{fig:corner_2d_deg2},~\ref{fig:corner_2d_deg2_cg}, and~\ref{fig:corner_2d_deg2_gmres}.

For $deg=3$ and higher (only $deg=3$ is shown here; see~\cite{mulita2019} for the full set of simulations), Richardson iteration turns out to be a bad smoother for S-AFEM, unless further tuning of the relaxation parameter $\omega$ is performed. Although $J(u_h^\ell$) seems to exhibit the same behaviour of $J(u_h)$, as the smoothing iteration count $\ell$ increases, contrarily to what one might expect, the value of the estimator increases with increasing DOFs (Figure~\ref{fig:corner_2d_deg3}), showing that our selection of $\omega$ may not be correct for these problems, and we should estimate in a better way the spectral radius of the final matrix $A$, and modify $\omega$ accordingly.

On the other hand, both the CG method and the GMRES method turn out to be good smoothers for S-AFEM, without the need to tune any parameter, as evidenced in Figures~\ref{fig:corner_2d_deg2_cg},~\ref{fig:corner_2d_deg2_gmres},~\ref{fig:corner_2d_deg3_cg},~\ref{fig:corner_2d_deg3_gmres}. In all cases, in fact, $J(u_h^\ell)$, for $\ell\ge 2$, shows the same optimal convergence rate as the $J(u_h)$ obtained by the classical AFEM, and although the total error at the intermediate cycles is evident, the accuracy of the final approximations is almost the same. 

\confrontoTwoD{corner_2d}{1}{}{Richardson}{corner problem in 2D}

\confrontoTwoD{corner_2d}{1}{_cg}{CG}{corner problem in 2D}

\confrontoTwoD{corner_2d}{1}{_gmres}{GMRES}{corner problem in 2D}

\confrontoTwoD{corner_2d}{2}{}{Richardson}{corner problem in 2D}
\confrontoTwoD{corner_2d}{2}{_cg}{CG}{corner problem in 2D}
\confrontoTwoD{corner_2d}{2}{_gmres}{GMRES}{corner problem in 2D}

\confrontoTwoD{corner_2d}{3}{}{Richardson}{corner problem in 2D}
\confrontoTwoD{corner_2d}{3}{_cg}{CG}{corner problem in 2D}
\confrontoTwoD{corner_2d}{3}{_gmres}{GMRES}{corner problem in 2D}

\afterpage{\clearpage}

\subsection{Two dimensional drift-diffusion problem}
\label{sec:safemdiffuciontransport}
We proceed with testing the accuracy of S-AFEM for a class of drift-diffusion problems where the transport term $\boldsymbol{\beta}$ is non-negligible. We consider a two-dimensional problem, where $\boldsymbol{\beta}=(\beta, \beta)^T$, and the scalar parameter $\beta$ takes the values $1,10$ and $50$. In particular, we impose boundary conditions and forcing terms so that the exact solution is
\begin{equation}\label{eq:diffusionetrasporto-exact}
    x + y + \frac{- e^{\beta x} + 1}{e^{\beta} - 1} - \frac{e^{\beta y} - 1}{e^{\beta} - 1}.
\end{equation}

Since the linear system associated to Problem~\eqref{eq:diffusionetrasportogenerale} when $\boldsymbol{\beta}\neq0$ is not symmetric, we use as a solver and as a smoother the GMRES method (cf. eg. \cite{liesen2013krylov}). We apply classic AFEM and S-AFEM for both bi-linear and higher order finite element discretizations for $deg=1,2,3$. We plot the value of the estimator $J$ and the value of the $|\cdot|_1$ semi-norm of the total error, using a fixed number of GMRES iterations as a smoother, with $\ell=1,3,5$.

For all the cases where the transport term $\boldsymbol{\beta} = (1, 1)^T$ is small, the behaviour of the estimator for S-AFEM is exactly the same as the one given by AFEM, for the case $deg=1$, as shown in Figure~\ref{fig:drift_diffusion_2d_beta1_deg1_gmres}, while it approaches it as the GMRES smoothing iteration count increases for higher order FEM discretizations (i.e. for $deg\ge2$), as shown in Figures~\ref{fig:drift_diffusion_2d_beta1_deg2_gmres}, and~\ref{fig:drift_diffusion_2d_beta1_deg3_gmres}. However, in all cases, the accuracy of the final approximations is almost the same. 

For the choices of the transport term $\boldsymbol{\beta} = (10, 10)^T$ (corresponding to moderate transport) and $\boldsymbol{\beta} = (50, 50)^T$ (corresponding to large transport) again the behaviour of the estimator for S-AFEM is exactly the same as the one given by AFEM for the case $deg=1$, as shown in Figures~\ref{fig:drift_diffusion_2d_beta10_deg1_gmres} and~\ref{fig:drift_diffusion_2d_beta50_deg1_gmres}, while it approaches it as the GMRES smoothing iteration count increases for higher order FEM discretizations, as evidenced in Figures~\ref{fig:drift_diffusion_2d_beta10_deg2_gmres},~\ref{fig:drift_diffusion_2d_beta10_deg3_gmres},~\ref{fig:drift_diffusion_2d_beta50_deg2_gmres}, and~\ref{fig:drift_diffusion_2d_beta50_deg3_gmres}. In all cases, however, the accuracy of the final approximations is almost the same, showing that S-AFEM turns out to be a good method also for drift-diffusion problems.

\confrontoTwoD{drift_diffusion_2d_beta1}{1}{_gmres}{GMRES}{drift-diffusion problem in 2D, with transport $\boldsymbol{\beta}=(1,1)$}
\confrontoTwoD{drift_diffusion_2d_beta1}{2}{_gmres}{GMRES}{drift-diffusion problem in 2D, with transport $\boldsymbol{\beta}=(1,1)$}
\confrontoTwoD{drift_diffusion_2d_beta1}{3}{_gmres}{GMRES}{drift-diffusion problem in 2D, with transport $\boldsymbol{\beta}=(1,1)$}

\confrontoTwoD{drift_diffusion_2d_beta10}{1}{_gmres}{GMRES}{drift-diffusion problem in 2D, with transport $\boldsymbol{\beta}=(10,10)$}
\confrontoTwoD{drift_diffusion_2d_beta10}{2}{_gmres}{GMRES}{drift-diffusion problem in 2D, with transport $\boldsymbol{\beta}=(10,10)$}
\confrontoTwoD{drift_diffusion_2d_beta10}{3}{_gmres}{GMRES}{drift-diffusion problem in 2D, with transport $\boldsymbol{\beta}=(10,10)$}

\confrontoTwoD{drift_diffusion_2d_beta50}{1}{_gmres}{GMRES}{drift-diffusion problem in 2D, with transport $\boldsymbol{\beta}=(50,50)$}        
\confrontoTwoD{drift_diffusion_2d_beta50}{2}{_gmres}{GMRES}{drift-diffusion problem in 2D, with transport $\boldsymbol{\beta}=(50,50)$}        
\confrontoTwoD{drift_diffusion_2d_beta50}{3}{_gmres}{GMRES}{drift-diffusion problem in 2D, with transport $\boldsymbol{\beta}=(50,50)$}        
        
\afterpage{\clearpage}

\subsection{Computational costs}

In the following table we show a comparison of the computational cost associated to the classical AFEM and to the smoothed AFEM in the pure diffusion case, for the first four examples we presented in this section. 

The results were obtained on a 2.8 GHz Intel Core i7 with 4 cores and 16GB of RAM, using MPI parallelization on all 4 cores.

\begin{table}
    \centering
    \resizebox{\textwidth}{!}{
    \begin{tabular}{l|r|r|r|r}    
       & Peak 2D & L-shaped 2D & Peak 3D & Fichera 3D \\
       \hline
       First and last solve (same for AFEM and S-AFEM) &0.0187s  & 0.0601s & 32s & 101s \\
       \hline
       \phantom{S-}AFEM Intermediate solves (CG) & 0.0663s &0.219s  & 76.4s &185s \\
       \hline
       S-AFEM Intermediate smoothing steps (Richardson)  & 0.005s & 0.00892s &  0.252s & 0.426s\\
       \hline
       S-AFEM intermediate speedup & 13.26 & 24.6 & 303.7 & 434.3 \\
       \hline
       S-AFEM total speedup & 3.59 & 4.045 & 3.361 & 2.819 
    \end{tabular}
    }
    \caption{Comparison of the computational cost of the solution stage for ten cycles of adaptive refinement using classical AFEM and S-AFEM on bi-linear elements, for the pure diffusion case.}
    \label{tab:computational-cost}
\end{table}

Table~\ref{tab:computational-cost} only shows the comparison between AFEM and S-AFEM in the solve phase, where S-AFEM is always faster than AFEM, offering an average speedup of a factor three. In the table we compare the computational cost of all intermediate cycles in S-AFEM (Intermediate smoothing steps (Richardson) in the table), with the corresponding computational cost for standard AFEM (Intermediate solves (CG) in the table). The first and last solve are the same in the two algorithms, and are reported to provide a scaling with respect to the total computational cost of the solution phase in the program. Other phases (like graphical output, mesh setup, assembling setup, and error estimation) are not shown since they are identical in the two algorithms.

\section{Conclusions}
\label{conclusions}

This work proposes a new smoothed algorithm for adaptive finite element methods (S-AFEM), inspired by multilevel techniques, where the exact algebraic solution in intermediate steps is replaced by the application of a prolongation step, followed by a fixed number of smoothing steps. 

The main argument behind the S-AFEM algorithm is that the combined application of the \Estimate-\Mark steps of AFEM is largely insensitive to substantial algebraic errors in low frequencies. Indeed, even though the intermediate solutions produced by S-AFEM are far from the exact algebraic solutions, we show that their a posteriori error estimation produces a refinement pattern for each cycle that is substantially equivalent to the one that would be generated by classical AFEM, leading roughly to the same set of cells marked for refinement. 

Our strategy is based on solving exactly the problem at the coarsest level and at the finest level, and then applying the \Estimate-\Mark-\Refine steps directly to the result of a smoother (\Smooth) in intermediate levels. 
 
We provide numerical evidences that the S-AFEM strategy is competitive in cost and accuracy by considering some variants of our algorithm, where different smoothers are used in the intermediate cycles (respectively Richardson iteration, the CG method, and the GMRES method). 

We conclude that, in general, CG and GMRES act as robust smoothers in S-AFEM also for high order approximations, and for non-symmetric problems, like, for example, drift-diffusion problems with dominant transport. 

Our numerical evidences show that two or three smoothing iterations are enough for the two-dimensional case, while three-dimensional problems require from five up to seven smoothing iteration in order to produce good final approximations, independently on the polynomial degree of the finite element approximation.

We analyzed the error propagation properties of the S-AFEM algorithm, and provided a bound on the a-posteriori error estimator applied to the approximated algebraic solution. The results are not sharp, and do not provide a definitive answer on the convergence of the final S-AFEM solution to the AFEM one, but could be used as a ground state for further investigation, which is currently ongoing.

\section*{Acknowledgments}

LH is partially supported by the Italian Ministry of Instruction, University and Research (MIUR), under the 2017 PRIN project NA-FROM-PDEs MIUR PE1, ``Numerical Analysis for Full and Reduced Order Methods for the efficient and accurate solution of complex systems governed by Partial Differential Equations''. OM is thankful to Durham University for the hospitality during her visiting research there.

\bibliographystyle{abbrv}
\bibliography{bibliography}
\end{document}